\documentclass[12pt,reqno]{amsart}
\usepackage{a4wide,amsfonts,amsmath,latexsym,amssymb,euscript,eufrak,graphicx,units,mathrsfs}
\usepackage[utf8]{inputenc}
\usepackage{amsmath}
\usepackage{amsfonts}
\usepackage{amssymb}
\usepackage{amsthm}
\usepackage{floatrow}
\usepackage{blindtext}
\usepackage{multicol}
\usepackage[english]{babel}
\usepackage{enumerate}
\usepackage{eufrak}
\usepackage{graphicx}
\usepackage{caption}
\usepackage{subcaption}
\usepackage{float}
\usepackage{epstopdf}
\usepackage{multirow}
\usepackage{mathtools}
\usepackage{multirow}
\usepackage[usenames,dvipsnames]{color}

\numberwithin{equation}{section}

\newcommand{\R}{\mathbb{R}}
\newcommand{\HC}{\mathcal{H}}

\newcommand{\stepid}[1]{\mathbf{1}_{[0, #1]}}

\newcommand{\norm}[1]{\left\lVert#1\right\rVert}
\newcommand{\inner}[2]{\langle #1, #2 \rangle}

\newtheorem{theorem}{Theorem}[section]
\newtheorem{lemma}[theorem]{Lemma}

\theoremstyle{definition}

\theoremstyle{remark}
\newtheorem{remark}[theorem]{Remark}

\numberwithin{equation}{section} 

\begin{document}
\title[Wasserstein bounds in  CLT of
AMCE and AMLE of the drift parameter for OU ]{Wasserstein bounds in
CLT of approximative MCE and MLE of the drift parameter for
Ornstein-Uhlenbeck processes observed at high frequency}
\thanks{\\
%Khalifa Es-Sebaiy (khalifa.essebaiy@ku.edu.kw), %\\ ORCID ID: 0000-0003-4915-6850
%Mishari Al-Foraih (mishari.alforaih@ku.edu.kw) and Fares Alazemi
%(fares.alazemi@ku.edu.kw)\\  Kuwait University, Faculty of Science,
%Department of Mathematics, Kuwait City, Kuwait.\\
\\This project was funded by Kuwait
Foundation for the Advancement of Sciences (KFAS) under project
code: PR18-16SM-04.}
\author[K. Es-Sebaiy]{ Khalifa Es-Sebaiy}
\address{Department of Mathematics, Faculty of Science, Kuwait University, Kuwait}
\email{khalifa.essebaiy@ku.edu.kw}
\author[F. Alazemi]{Fares Alazemi}
\address{Department of Mathematics, Faculty of Science, Kuwait University, Kuwait}
\email{fares.alazemi@ku.edu.kw}
\author[M. Al-Foraih]{Mishari Al-Foraih}
\address{Department of Mathematics, Faculty of Science, Kuwait University, Kuwait}
\email{mishari.alforaih@ku.edu.kw}

\begin{abstract} This paper deals with the rate of convergence for the central limit
theorem of estimators  of the drift coefficient, denoted $\theta$,
for a Ornstein-Uhlenbeck process $X \coloneqq \{X_t,t\geq0\}$
observed at high frequency. We provide an Approximate minimum
contrast estimator and an approximate maximum likelihood  estimator
of $\theta$, namely  $\widetilde{\theta}_{n}\coloneqq
{1}/{\left(\frac{2}{n} \sum_{i=1}^{n}X_{t_{i}}^{2}\right)}$,
  and $\widehat{\theta}_{n}\coloneqq -{\sum_{i=1}^{n}
X_{t_{i-1}}\left(X_{t_{i}}-X_{t_{i-1}}\right)}/{\left(\Delta_{n}
\sum_{i=1}^{n} X_{t_{i-1}}^{2}\right)}$, respectively, where $ t_{i}
= i \Delta_{n}$,
  $ i=0,1,\ldots, n $, $\Delta_{n}\rightarrow 0$.
   We   provide Wasserstein bounds in
    central limit theorem for $\widetilde{\theta}_{n}$ and $\widehat{\theta}_{n}$.
\end{abstract}

\maketitle



\allowdisplaybreaks

\renewcommand{\thefootnote}{\arabic{footnote}}

%%%%%%%%%%%%%%%%%%%%%%%%%%%%%
% Section: Introduction
%%%%%%%%%%%%%%%%%%%%%%%%%%%%%
\section{Introduction}

Let $X:=\left\{X_t, t \geq 0\right\}$  be the Ornstein-Uhlenbeck
(OU) process driven by a   Brownian motion $\left\{W_{t},t\geq
0\right\} $.  More precisely,  $X$ is the solution of the following
linear stochastic differential equation
\begin{equation}
X_{0}=0;\quad dX_{t}=-\theta X_{t}dt+dW_{t},\quad t\geq 0,
\label{INTRO-OU}
\end{equation}
where $\theta >0$ is an unknown parameter.\\
The drift parametric estimation for the OU process \eqref{INTRO-OU}
has been widely studied in the literature.
 There are several  methods that can estimate the parameter $\theta$
 in \eqref{INTRO-OU} such as maximum likelihood estimation, least squares
 estimation and  minimum
contrast estimation, we refer to  monographs \cite{kutoyants,LS}.
While for the study of the asymptotic distribution of the estimators
of $\theta$ based on discrete observations of $X$, there is
extensive literature, only several works have been dedicated to the
rates   of weak convergence of the distributions of the estimators
to the standard normal distribution.

From a practical point of view, in parametric inference, it is more
realistic and interesting to consider asymptotic estimation for
\eqref{INTRO-OU} based on discrete observations. Thus, let us assume
that the process $X$ given in \eqref{INTRO-OU}, is observed
equidistantly in time with the step size $\Delta_n$ : $t_i=i
\Delta_n, i=0, \cdots, n$, and $T=n \Delta_n$ denotes the length of
the "observation window". Here we are concerned with the approximate
minimum contrast estimator (AMCE)
\[\widetilde{\theta}_{n}\coloneqq \frac{1}{\frac{2}{n}
\sum_{i=1}^{n}X_{t_{i}}^{2}},\]
 and
the approximate maximum likelihood  estimator (AMLE)
\[\widehat{\theta}_{n}\coloneqq -\frac{\sum_{i=1}^{n}
X_{t_{i-1}}\left(X_{t_{i}}-X_{t_{i-1}}\right)}{\Delta_{n}
\sum_{i=1}^{n} X_{t_{i-1}}^{2}},\] which are
 discrete versions of the minimum contrast estimator
(MCE) and the maximum likelihood  estimator (MLE) defined,
respectively,  as follows:
\[
\bar{\theta}_{T}:=\frac{1}{2\int_{0}^{T} X_{s}^{2} \mathrm{~d}
s},\qquad\qquad  \check{\theta}_{T}=\frac{\int_{0}^{T} X_{s}
\mathrm{~d} X_{s}}{\int_{0}^{T} X_{s}^{2} \mathrm{~d} s}, \quad
T\geq 0.
\]
Recall that, for two random variables $X$ and $Y$, the Wasserstein
metric is given by
\begin{eqnarray*}
 d_{W}\left( X,Y\right) \coloneqq \sup_{f\in
Lip(1)}\left\vert E [f(X)]-E [f(Y)]\right\vert,
\end{eqnarray*}
where $Lip(1)$ is the set of all Lipschitz functions with Lipschitz
constant $\leqslant 1$.

 Rates of convergence in the central limit
theorem of the MCE $\bar{\theta}_{T}$ and MLE $\check{\theta}_{T}$
under the Kolmogorov and Wasserstein distances have been studied as
follows: There exist $c,\ C>0$ depending only on $\theta$ such that
 \[\sup _{x \in
\mathbb{R}}\left|P\left(\sqrt{\frac{T}{2
\theta}}\left(\bar{\theta}_{T}-\theta\right) \leqslant
x\right)-P\left(\mathcal{N}\leqslant x\right)\right|\leq
\frac{C}{\sqrt{T}},   \mbox{ see \cite[Theorem 2.5]{bishwal2010},}\]
 \[
d_{W}\left(\sqrt{\frac{T}{2\theta}}\left(\bar{\theta}_{T}-\theta\right),
N\right) \leq  \frac{C}{\sqrt{T}}, \mbox{ see \cite[Theorem
5.4]{DEKN}},
\]
\[\frac{c}{\sqrt{T}}\leq\sup _{x \in
\mathbb{R}}\left|P\left(\sqrt{\frac{T}{2
\theta}}\left(\check{\theta}_{T}-\theta\right) \leqslant
x\right)-P\left(\mathcal{N}\leqslant x\right)\right|\leq
\frac{C}{\sqrt{T}},  \mbox{ see \cite[Theorems 1 and 2]{KP},}\]
 \[
d_{W}\left(\sqrt{\frac{T}{2\theta}}\left(\check{\theta}_{T}-\theta\right),
\mathcal{N}\right) \leq  \frac{C}{\sqrt{T}}, \mbox{ see
\cite[Theorem 1]{EAA} for fixed $N=1$},
\]
 where   $\mathcal{N}\sim
\mathcal{N}\left(0,1\right)$ denotes a standard normal random
variable.

 The purpose of this manuscript is to derive upper
bounds of the Wasserstein distance for the rates of convergence of
the distribution of the AMCE $\widetilde{\theta}_{n}$ and  the AMLE
$\widehat{\theta}_{n}$. These estimators are unbiased and we show
that they are  consistent and admit a central limit theorem as
$\Delta_n\rightarrow0$ and $T\rightarrow\infty$. Moreover, we bound
the rate of convergence to the normal distribution in terms of
Wasserstein distance.

Note that the papers \cite{bishwal2006} and \cite{BB}  provided
explicit upper bounds for the Kolmogorov  distance for the rates of
convergence of the distribution of
   $\widetilde{\theta}_{n}$ and  $\widehat{\theta}_{n}$, respectively. On the other
   hand, \cite{DEKN} provided  Wasserstein bounds in
    central limit theorem for $\widetilde{\theta}_{n}$.
   Let us describe what is proved  in this
 direction:
\begin{itemize}
\item  Theorem 2.1 in \cite{bishwal2006}
shows that  there exists $C>0$ depending on $\theta$ such that
 \begin{eqnarray}\sup _{x \in
\mathbb{R}}\left|P\left(\sqrt{\frac{T}{2
\theta}}\left(\widetilde{\theta}_{n}-\theta\right) \leqslant
x\right)-P\left(\mathcal{N}\leqslant x\right)\right|\leq C \max
\left(\sqrt{\frac{\log T}{T}},\frac{T^4}{n^2\log
T}\right).\label{bishwal-eq1}\end{eqnarray}
\item    Theorem 2.3 in \cite{BB} proves that there exists  $C>0$ depending on $\theta$ such that
\begin{eqnarray}\sup _{x \in
\mathbb{R}}\left|P\left(\sqrt{\frac{T}{2
\theta}}\left(\widehat{\theta}_{n}-\theta\right) \leqslant
x\right)-P\left(\mathcal{N}\leqslant x\right)\right|\leq C \max
\left(\sqrt{\frac{\log T}{T}},\frac{T^2}{n\log
T}\right).\label{bishwal-eq2}\end{eqnarray}
\item  Theorem 5.4 in \cite{DEKN} establishes that there exists  $C>0$ depending on $\theta$ such that
\begin{eqnarray}
d_{W}\left(\sqrt{\frac{T}{2\theta}}\left(\widetilde{\theta}_{n}-\theta\right),
\mathcal{N}\right) \leq  C\max
\left(\frac{1}{\sqrt{T}},\sqrt{\frac{T^2}{n}}\right).\label{DEKN-eq}
\end{eqnarray}
 \end{itemize}

\begin{remark} Note that in \cite[Theorem 2.1]{bishwal2006},   \cite[Theorem 2.3]{BB} and  \cite[Theorem 5.4]{DEKN}, the asymptotic normality  of the
distribution of   $\widetilde{\theta}_{n}$ and
$\widehat{\theta}_{n}$ need $n\Delta_n^2=\frac{T^2}{n} \rightarrow
0$ and $T \rightarrow \infty$. However, Theorem
\ref{rate-CLT-theta-tilde} and Theorem \ref{thm:sigma-hat},  which
are stated and proved below, show that, respectively, the asymptotic
normality  of the distribution of $\widetilde{\theta}_{n}$ and
$\widehat{\theta}_{n}$ only need $\Delta_n=\frac{T}{n} \rightarrow
0$ and $T \rightarrow \infty$.
\end{remark}

The aim of the present paper is to provide new explicit bounds for
the rate of convergence in the CLT of the estimators
$\widetilde{\theta}_{n}$ and $\widehat{\theta}_{n}$
 under the Wasserstein metric as follows: There exists a constant $C>0$ such that, for all $n\geq1,\ T>0$,
   \begin{eqnarray}
d_{W}\left(\sqrt{\frac{T}{2\theta}}\left(\widetilde{\theta}_{n}-\theta\right),
\mathcal{N}\right) \leq  C
\max\left(\frac{1}{\sqrt{T}},\frac{T^2}{n^2}\right),
\label{intro-eq:bound_tilde}
\end{eqnarray}
see Theorem \ref{rate-CLT-theta-tilde}, and
  \begin{eqnarray}
 d_W\left(\sqrt{\frac{T}{2\theta}}\left(\widehat{\theta}_{n}-\theta\right),\mathcal{N}\right)
 &\leq&C\left(\frac{1}{\sqrt{T}} , \sqrt{\frac{T^3}{n^2}}\right), \label{intro-rate-new-hat}
\end{eqnarray}
see Theorem \ref{thm:sigma-hat}.

\begin{remark}The estimates \eqref{intro-eq:bound_tilde} and \eqref{intro-rate-new-hat} show that we have
  improved the bounds on the error of normal approximation for $\widetilde{\theta}_{n}$ and
  $\widehat{\theta}_{n}$. In other words,
 it is clear that the obtained bounds in \eqref{intro-eq:bound_tilde} and \eqref{intro-rate-new-hat} are sharper than the bounds in
\eqref{bishwal-eq1}, \eqref{bishwal-eq2} and \eqref{DEKN-eq}.
\end{remark}

To finish this introduction, we note the general structure of this
paper. Section 2 contains some preliminaries presenting the tools
needed  from the analysis on Wiener space, including Wiener chaos
calculus and Malliavin calculus. Upper bounds for the rates of
convergence of the distribution of  the AMCE
$\widetilde{\theta}_{n}$ and  the AMLE $\widehat{\theta}_{n}$ are
provided in Section 3 and Section 4, respectively.

\section{Preliminaries}
This section gives a brief overview of some useful facts from the
Malliavin calculus on Wiener space. Some of the results presented
here are essential for the proofs in the present paper. For our
purposes we focus on special cases that are relevant for our setting
and omit the general high-level theory. We direct the interested
reader to~\cite[Chapter 1]{nualart-book}and~\cite[Chapter
2]{NP-book}.

Fix $\left( \Omega ,\mathcal{F},P \right) $ for the Wiener space of
a standard Wiener process $W = (W_t)_{t \geq 0}$. The first step is
to identify the general centered Gaussian process $(Z_t)_{\geq 0}$
with an \emph{isonormal Gaussian process} $X = \{ X(h), h \in
\mathcal{H}\}$ for some Hilbert space $\mathcal{H}$. Recall that for
such processes $X$, for every $h_1, h_2 \in \mathcal{H}$, one has
$E [ X(h_1) X(h_2) ] = \inner{h_1}{h_2}_\HC$.\\
One can define $\HC$ as the closure of real-valued step functions on
$[0, \infty)$ with respect to the inner product
$\inner{\stepid{t}}{\stepid{s}}_\HC = E[ Z_t Z_s]$. Then the
isonormal process $X$ is given by Wiener integral $X(h) \coloneqq
\int_{\R^+} h(s) dW_s$. Note, that, in particular $X(\stepid{t})
\overset{d}{=} Z_t$.\\
The next step involves the \emph{multiple Wiener-It\^o integrals}.
The formal definition involves the concepts of Malliavin derivative
and divergence. We refer the reader to~\cite[Chapter
1]{nualart-book}and~\cite[Chapter 2]{NP-book}. For our purposes we
define the multiple Wiener-It\^o integral $I_p$ via the Hermite
polynomials $H_p$. In particular, for $h \in \HC$ with $\norm{h}_\HC
= 1$, and any $p \geq 1$,
\begin{align}
\notag H_p(X(h)) = I_p(f^{\otimes p}).
\end{align}
For $p = 1$ and $p = 2$ we have the following:
\begin{align}
\label{eq:z_rep} H_1(X(\stepid{t})) = & X(\stepid{t}) = I_1(\stepid{t}) = Z_t \\
\label{eq:var_rep} H_2(X (\stepid{t})) = & X(\stepid{t})^2 -
E[X(\stepid{t})^2] = I_2(\stepid{t}^{\otimes 2}) = Z_t^2 - E[Z_t]^2.
\end{align}
Note also that $I_0$ can be taken to be the identity operator.

\noindent $\bullet $ \textbf{Some notation for Hilbert spaces.} Let
$\HC$ be a Hilbert space. Given an integer $q \geq 2$ the Hilbert
spaces $\HC^{\otimes q}$ and $\HC^{\odot q}$ correspond to the $q$th
\emph{tensor product} and $q$th \emph{symmetric tensor product} of
$\HC$. If $f \in \HC^{\otimes q}$ is given by $f = \sum_{j_1,
\ldots, j_q} a(j_1, \ldots, j_q) e_{j_1} \otimes \cdots e_{j_q}$,
where $(e_{j_i})_{i \in [1, q]}$ form an orthonormal basis of
$\HC^{\otimes q}$, then the symmetrization $\tilde{f}$ is given by
\begin{align}
\notag \tilde{f} = \frac{1}{q!} \sum_{\sigma} \sum_{j_1, \ldots,
j_q} a(j_1, \ldots, j_q) e_{\sigma(j_1)} \otimes \cdots
e_{\sigma(j_q)},
\end{align}
where the first sum runs over all permutations  $\sigma$ of $\{1,
\ldots, q\}$. Then $\tilde{f}$ is an element of $\HC^{\odot q}$. We
also make use of the concept of contraction. The $r$th
\emph{contraction} of two tensor products $e_{j_1} \otimes \cdots
\otimes e_{j_p}$ and $e_{k_1} \otimes \cdots e_{k_q}$ is an element
of $\HC^{\otimes (p + q - 2r)}$ given by
\begin{align}
\notag (e_{j_1} & \otimes \cdots \otimes e_{j_p}) \otimes_r (e_{k_1} \otimes \cdots \otimes e_{k_q}) \\
\label{eq:contraction} =  & \quad \left[ \prod_{\ell =1}^r
\inner{e_{j_\ell}}{e_{k_\ell}} \right] e_{j_{r+1}} \otimes \cdots
\otimes e_{j_q} \otimes e_{k_{r+1}} \otimes \cdots \otimes e_{k_q}.
\end{align}

\noindent $\bullet $ \textbf{Isometry property of
integrals~\cite[Proposition 2.7.5]{NP-book}} Fix integers $p, q \geq
1$ as well as $f \in \HC^{\odot p}$ and $g \in \HC^{\odot q}$.
\begin{align}
\label{eq:isometry}  E [ I_q(f) I_q(g) ] = \left\{
\begin{array}{ll} p! \inner{f}{g}_{\HC^{\otimes p}} & \mbox{ if } p
= q \\ 0 & \mbox{otherwise.} \end{array} \right.
\end{align}

\noindent $\bullet $ \textbf{Product formula~\cite[Proposition
2.7.10]{NP-book}} Let $p,q \geq 1$. If $f \in \HC^{\odot p}$ and $g
\in \HC^{\odot q}$ then
\begin{align}
\label{eq:product} I_p(f) I_q(g) = \sum_{r = 0}^{p \wedge q} r! {p
\choose r} {q \choose r} I_{p + q -2r}(f \widetilde{\otimes}_r g).
\end{align}

\noindent $\bullet $ \textbf{Hypercontractivity in Wiener Chaos.}
For every $q\geq 1$, ${\mathcal{H}}_{q}$ denotes the $q$th Wiener
chaos of $W$, defined as the closed linear subspace of $L^{2}(\Omega
)$ generated by the random variables $\{H_{q}(W(h)),h\in
{{\mathcal{H}}},\Vert h\Vert _{{\mathcal{H}}}=1\}$ where $H_{q}$ is
the $q$th Hermite polynomial. For any $F \in
\oplus_{l=1}^{q}{\mathcal{H}}_{l}$ (i.e. in a fixed sum of Wiener
chaoses), we have
\begin{equation}
\left( E\big[|F|^{p}\big]\right) ^{1/p}\leqslant c_{p,q}\left(
E\big[|F|^{2}\big]\right) ^{1/2}\ \mbox{ for any }p\geq 2.
\label{hypercontractivity}
\end{equation}
It should be noted that the constants $c_{p,q}$ above are known with
some precision when $F$ is a single chaos term: indeed,
by~\cite[Corollary 2.8.14]{NP-book}, $c_{p,q}=\left( p-1\right)
^{q/2}$.

\noindent $\bullet $ \textbf{Optimal fourth moment theorem}. Let $N$
denote the standard normal law.
Let a sequence $X:X_{n}\in {\mathcal{H}}_{q}$, such that $E%
X_{n}=0$ and $Var\left[ X_{n}\right] =1$ , and assume $X_{n}$
converges to a normal law in distribution, which is equivalent to
$\lim_{n}E\left[ X_{n}^{4}\right] =3$. Then we have the  optimal
estimate for total variation distance $d_{TV}\left(
X_n,\mathcal{N}\right) $, known as the optimal 4th moment theorem,
proved in \cite{NP2015}. This optimal estimate also holds with
Wasserstein distance $d_{W}\left( X_n,\mathcal{N}\right) $, see
\cite[Remark 2.2]{DEKN}, as follows: there exist two constants
$c,C>0$ depending only on the sequence $X$ but not on $n$, such that
\begin{equation}
c\max \left\{ E\left[ X_{n}^{4}\right] -3,\left\vert E%
\left[ X_{n}^{3}\right] \right\vert \right\} \leqslant d_{W}\left(
X_n,\mathcal{N}\right) \leqslant C\max \left\{ E\left[
X_{n}^{4}\right] -3,\left\vert E\left[ X_{n}^{3}\right] \right\vert
\right\}.\label{fourth-cumulant-thm}
\end{equation}
 Moreover, we recall that   the third and fourth
cumulants are respectively
$$
\begin{aligned}
&\kappa_{3}(X)=E\left[X^{3}\right]-3 E\left[X^{2}\right] E[X]+2 E[X]^{3} \\
&\kappa_{4}(X)=E\left[X^{4}\right]-4 E[X] E\left[X^{3}\right]-3
E\left[X^{2}\right]^{2}+12 E[X]^{2} E\left[X^{2}\right]-6 E[X]^{4}.
\end{aligned}
$$
In particular, when $E[X]=0$, we have that
\[\kappa_{3}(X)=E\left[X^{3}\right]\ \mbox{ and }\ \kappa_{4}(X)=
E\left[X^{4}\right]-3 E\left[X^{2}\right]^{2}.\]
 If
${g\in\mathcal{H}^{\otimes 2}}$, then the third and fourth cumulants
for $I_2(g)$ satisfy the following (see (6.2) and (6.6) in
\cite{BBNP}, respectively),
\begin{eqnarray}
k_{3}(I_2(g))  =E[(I_2(g))^3] =8\left<g,g\otimes_1
g\right>_{\mathcal{H}^{\otimes 2}},\label{3rd-cumulant}
\end{eqnarray}
and
\begin{eqnarray} \left|k_{4}(I_2(g))\right|
&=&16\left(\|g\otimes_1 g\|_{\mathcal{H}^{\otimes
2}}^2+2\|g\widetilde{\otimes_1} g\|_{\mathcal{H}^{\otimes
2}}^2\right)\nonumber\\&\leq&48\|g\otimes_1
g\|_{\mathcal{H}^{\otimes 2}}^2.\label{4th-cumulant}
\end{eqnarray}

\begin{lemma}[\cite{NZ}]\label{NZ-lemma} Fix an integer $M \geq 2 .$ We have
\[
\sum_{\left|k_{j}\right| \leq n \atop 1 \leq j \leq
M}|\rho(\mathbf{k} \cdot \mathbf{v})|
\prod_{j=1}^{M}\left|\rho\left(k_{j}\right)\right| \leq
C\left(\sum_{|k| \leq n}|\rho(k)|^{1+\frac{1}{M}}\right)^{M}
\]
where $\mathbf{k}=\left(k_{1}, \ldots, k_{M}\right)$ and $\mathbf{v}
\in \mathbb{R}^{M}$ is a fixed vector whose components are 1 or -1.
\end{lemma}

Throughout the paper $\mathcal{N}$ denotes a standard normal random
variable. Also, $C$ denotes a generic positive constant (perhaps
depending on $\theta$, but not on anything else), which may change
from line to line.

%%%%%%%%%%%%%%%%%%%%%%%%%%%%%
% Section: Stationary
%%%%%%%%%%%%%%%%%%%%%%%%%%%%%
\section{Approximate minimum contrast estimator}

In this section we  prove the consistency and provide upper bounds
in the  Wasserstein distance  for the rate of normal convergence of
an approximate minimum contrast estimator of the drift parameter
$\theta$ of the Ornstein-Uhlenbeck process $X \coloneqq
\left\{X_{t},t\geq 0\right\} $ driven by a   Brownian motion
$\left\{W_{t},t\geq 0\right\} $, defined as  solution of the
following linear stochastic differential equation
\begin{equation}
X_{0}=0;\quad dX_{t}=-\theta X_{t}dt+dW_{t},\quad t\geq 0,\label{OU}
\end{equation}
where $\theta >0$ is an unknown parameter. Since \eqref{OU} is
linear, it is immediate to see that its solution can be expressed
  explicitly as
\begin{align}
X_{t}=\int_{0}^{t}e^{-\theta (t-s)}dW_{s}. \label{OUX}
\end{align}
Moreover,
\begin{equation}
Z_{t}=\int_{-\infty }^{t}e^{-\theta (t-s)}dW_{s}\label{OUZ}
\end{equation}
 is a stationary Gaussian process, see \cite{CKM,EV}.\\
 Furthermore,
 \begin{equation}
X_{t}=Z_{t}-e^{-\theta t}Z_0.\label{X-Z}
\end{equation}

 Since $Z \coloneqq \{Z_{t},   t\geq0 \}$ is a continuous centered stationary Gaussian
 process, then
  it can be represented as a Wiener-It\^o (multiple)
   integral $Z_{t} = I_{1}(\stepid{t})$ for every $t\geq 0$, as  in~\eqref{eq:z_rep}.
    Let $\rho(r)=E(Z_rZ_0)$ denote the covariance of $Z$ for every
    $r\geq0$. It is easy to show that
\[\rho(t)=E(Z_tZ_0)=\frac{e^{-\theta |t|}}{2\theta},\quad t\in\mathbb{R}.\]
In particular, $\rho(0)=\frac{1}{2\theta}$. Moreover,
   notice that $\rho(r)=\rho(-r)$ for all $r<0$.

Our goal is to estimate $\theta$ based the  discrete observations of
$X$, using the approximative minimum contrast estimator:
\begin{eqnarray}
\begin{gathered}
\widetilde{\theta}_{n}:=\frac{1}{2\left(\frac{1}{n}
\sum_{i=1}^{n}X_{t_{i}}^{2}\right)}=\frac{1}{2f_{n}\left(X\right)}=g(f_{n}\left(X\right)),
\quad n \geq 1,
\end{gathered}\label{expr-theta-tilde}
\end{eqnarray}
where  $g(x):=\frac{1}{2x}$, $t_{i}=i \Delta_{n}, i=0, \ldots, n,
\Delta_{n} \rightarrow 0$ and $T=n \Delta_{n}$, whereas
$f_{n}\left(X\right),\ n\geq1$, are given by
\begin{align}
f_n(X) \coloneqq \frac{1}{n} \sum_{i =0}^{n-1} X_{t_{i}}^{2}.
\label{def-f_n}
\end{align}

In order to analyze the estimator $\widetilde{\theta}_{n}$ of
$\theta$ based on discrete high-frequency data in time of $X$, we
first estimate the limiting variance $\rho(0)=\frac{1}{2\theta}$  by
 the  estimator $f_{n}\left(X\right)$, given by \eqref{def-f_n}.\\
Let us introduce
\[F_n(Z):=\sqrt{T}\left(f_n(Z)-\frac{1}{2\theta}\right), \mbox{ where } f_n(Z) \coloneqq \frac{1}{n} \sum_{i =0}^{n-1}
Z_{t_{i}}^{2}.\] According to \eqref{eq:var_rep}, $F_n(Z)$ can be
written as
\begin{align}
F_n(Z) = \sqrt{\frac{\Delta_n}{n}} \sum_{i =0}^{n-1}
I_2(\stepid{t_i}^{\otimes 2})=I_2\left(\sqrt{\frac{\Delta_n}{n}}
\sum_{i =0}^{n-1} \stepid{t_i}^{\otimes 2}\right)=:
I_2(\varepsilon_n).\label{F_n(Z) in 2nd chaos}
\end{align}
By \eqref{X-Z}, straightforward  calculation leads to the following
technical lemma.
\begin{lemma} Let $X$ and $Z$ be the processes given in \eqref{OUX}
and \eqref{OUZ} respectively. Then there exists $C>0$ depending only
on $\theta$ such that  for every $p \geqslant 1$ and for all $n \in
\mathbb{N}$,
\begin{equation}
\left\|F_n(X)-F_n(Z)\right\|_{L^p(\Omega)}\leq
\frac{C}{n\Delta_n}.\label{error X-Z}
\end{equation}
\end{lemma}

\begin{lemma}
There exists $C>0$ depending only on $\theta$ such that for large
$n$
\begin{align}
\left|E\left(F_n^2(Z)\right)-\frac{1}{2\theta^3}\right|&\leq
C\left(\Delta_n^2+\frac{1}{n\Delta_n}\right).\label{variance-F(Z)}
\end{align}
Consequently, using \eqref{error X-Z}, for large $n$
\begin{align}
\left|E\left(F_n^2(X)\right)-\frac{1}{2\theta^3}\right|&\leq
C\left(\Delta_n^2+\frac{1}{n\Delta_n}\right).\label{variance-F(X)}
\end{align}
\end{lemma}

\begin{proof}
Using the well-known Wick formula, we have
\begin{align}E\left(Z_{t}^{2}Z_{s}^{2}\right)=E\left(Z_{t}^{2}\right)E\left(Z_{s}^{2}\right)+2\left(E\left(Z_{t}Z_{s}\right)\right)^2
=\rho^2(0)+2\rho^2(t-s).\label{wick-eq}\end{align} This implies
\begin{align}
E\left(F_n^2(Z)\right)&=T\left[Ef_n^2(Z)-2\frac{1}{2\theta}Ef_n(Z)+\rho^2(0)\right]\\
&=T\left[Ef_n^2(Z)-\rho^2(0)\right] \nonumber\\
&= T\left[\frac{1}{n^2} \sum_{i,j =0}^{n-1} E\left(Z_{t_{i}}^{2}Z_{t_{j}}^{2}\right)-\rho^2(0)\right] \nonumber \\
&= T\left[\frac{2}{n^2} \sum_{i,j =0}^{n-1}
\rho^2\left(t_{j}-t_{i}\right)\right] \nonumber\\
&= \frac{2\Delta_n}{n} \sum_{i,j =0}^{n-1}
\rho^2\left((j-i)\Delta_\mathcal{N}\right)=\frac{2\Delta_n}{n}
\sum_{i,j
=0}^{n-1}\frac{e^{-2\theta |j-i|\Delta_n}}{(2\theta)^2}\nonumber\\
&=\frac{\Delta_n}{2\theta^2}+  \frac{\Delta_n}{\theta^2n}
\sum_{0\leq i<j\leq n-1} e^{-2\theta (j-i)\Delta_n}\nonumber\\
&=\frac{\Delta_n}{2\theta^2}+  \frac{\Delta_n}{\theta^2n}
\sum_{k=1}^{n-1} (n-k)e^{-2k \Delta_n\theta}\nonumber\\
&=\frac{-\Delta_n}{2\theta^2}+  \frac{\Delta_n}{\theta^2}
\sum_{k=0}^{n-1} e^{-2k \Delta_n\theta }- \frac{\Delta_n}{\theta^2n}
\sum_{k=1}^{n-1} ke^{-2k \Delta_n\theta }.\label{estimate1}
\end{align}
Further,
\begin{align}
  \frac{-\Delta_n}{2\theta^2}+ \frac{\Delta_n}{\theta^2}
\sum_{k=0}^{n-1} e^{-2k \Delta_n\theta
}&=\frac{-\Delta_n}{2\theta^2}+\frac{\Delta_n}{\theta^2}\frac{1-e^{-2n\theta\Delta_n}}{1-e^{-2\theta\Delta_n}}\nonumber\\
&=\frac{-\Delta_n}{2\theta^2}+\frac{1}{\theta^2}\frac{\Delta_n}{1-e^{-2\theta\Delta_n}}
-\frac{1}{\theta^2}\frac{\Delta_n}{1-e^{-2\theta\Delta_n}}e^{-2n\theta\Delta_n}\nonumber\\
&=\frac{-\Delta_n}{2\theta^2}+\frac{1}{\theta^2}\frac{1}{2\theta(1-\theta\Delta_n
+o(\Delta_n))}
-\frac{1}{\theta^2}\frac{\Delta_n}{1-e^{-2\theta\Delta_n}}e^{-2\theta
n\Delta_n}\nonumber\\
&=\frac{-\Delta_n}{2\theta^2}+\frac{1}{2\theta^3}\left(1+\theta\Delta_n
+\theta^2\Delta_n^2+o(\Delta_n^2)\right)
-\frac{1}{\theta^2}\frac{\Delta_n}{1-e^{-2\theta\Delta_n}}e^{-2\theta
n\Delta_n}\nonumber\\
&=\frac{1}{2\theta^3}\left(1+\theta^2\Delta_n^2+o(\Delta_n^2)\right)
-\frac{1}{\theta^2}\frac{\Delta_n}{1-e^{-2\theta\Delta_n}}e^{-2\theta
n\Delta_n}.\label{estimate2}
\end{align}
Moreover,
\begin{align}
 \frac{\Delta_n}{\theta^2n}\sum_{k=1}^{n-1} ke^{-2k \Delta_n\theta}
 = \frac{1}{\theta^2n\Delta_n}\sum_{k=1}^{n-1} (k\Delta_n)e^{-2k \Delta_n\theta}\Delta_n,\label{estimate3}
\end{align}
and as $n\rightarrow\infty$
\begin{align*}
 \sum_{k=1}^{n-1} (k\Delta_n)e^{-2k
 \Delta_n\theta}\Delta_n\longrightarrow\int_0^{\infty}xe^{-2\theta
 x}dx=\frac{1}{2\theta^2}<\infty.
\end{align*}
Combining \eqref{estimate1}, \eqref{estimate2} and \eqref{estimate3}
and
$\frac{\Delta_n}{1-e^{-2\theta\Delta_n}}\rightarrow\frac{1}{2\theta}$,
there exists $C>0$ depending only on $\theta$ such that for large
$n$
\begin{align*}
\left|E\left(F_n^2(Z)\right)-\frac{1}{2\theta^3}\right|&\leq
C\left(\Delta_n^2+e^{-2\theta
n\Delta_n}+\frac{1}{n\Delta_n}\right)\\&\leq
C\left(\Delta_n^2+\frac{1}{n\Delta_n}\right).
\end{align*}
Therefore the desired result is obtained.
\end{proof}

\begin{lemma}There exists $C>0$ depending only on $\theta$ such that for
all $n\geq1$,
\begin{eqnarray}
|k_{3}(F_n(Z))|&\leq&C\frac{\Delta_n^{1/2}}{n^{3/2}},\label{k3-estimator1}
\end{eqnarray}
\begin{eqnarray} \left|k_{4}(F_n(Z))\right|\leq C\frac{1}{n\Delta_n}.\label{k4-estimator1}
\end{eqnarray}
Consequently,
\begin{eqnarray} \max\left(|k_{3}(F_n(Z))|,\left|k_{4}(F_n(Z))\right|\right)\leq C\frac{1}{n\Delta_n}.\label{max(k3,k4)-estimator1}
\end{eqnarray}
\end{lemma}
\begin{proof}
Using $\mathbf{1}_{[0, s]}^{\otimes 2} \otimes_{1} \mathbf{1}_{[0,
t]}^{\otimes 2}=\left\langle\mathbf{1}_{[0, s]}, \mathbf{1}_{[0,
t]}\right\rangle_{\mathcal{H}} \mathbf{1}_{[0, s]} \otimes
\mathbf{1}_{[0, t]}=\rho(t-s) \mathbf{1}_{[0, s]} \otimes
\mathbf{1}_{[0, t]},$ we can write
\begin{eqnarray*}
\varepsilon_n\otimes_1 \varepsilon_n=\frac{\Delta_n}{n}
\sum_{i,j=0}^{n-1} \rho(t_j-t_i)\mathbf{1}_{[0, t_i]} \otimes
\mathbf{1}_{[0, t_j]}.
\end{eqnarray*}
Combining this with \eqref{3rd-cumulant} and \eqref{F_n(Z) in 2nd
chaos},  we get
\begin{eqnarray}
k_{3}(F_n(Z))=k_{3}(I_2(\varepsilon_n))
&=&8\left<\varepsilon_n,\varepsilon_n\otimes_1
\varepsilon_n\right>_{\mathcal{H}^{\otimes 2}}\nonumber\\
&=&\frac{\Delta_n^{3/2}}{n^{3/2}} \sum_{i,j,k=0}^{n-1}
\rho(t_j-t_i)\rho(t_i-t_k)\rho(t_k-t_j)\nonumber\\
&=&\frac{\Delta_n^{3/2}}{n^{3/2}} \sum_{i,j,k=0}^{n-1}
\rho((j-i)\Delta_n)\rho((i-k)\Delta_n)\rho((k-j)\Delta_n)\nonumber\\
&\leq&\frac{\Delta_n^{3/2}}{n^{3/2}} \sum_{|k_i|< n,i=1,2,3}
\rho(k_1\Delta_n)\rho(k_2\Delta_n)\rho(k_3\Delta_n)\nonumber\\
&\leq&\frac{\Delta_n^{3/2}}{n^{3/2}} \left(\sum_{|k|< n}
\rho(k\Delta_n)\right)^3.\label{k3-estimate1}
\end{eqnarray}
On the other hand,
\begin{eqnarray}
 \sum_{|k|< n}\rho(k\Delta_n)&=&\frac{1}{2\theta}\sum_{|k|< n}e^{-\theta
 |k|\Delta_n}\nonumber\\
 &\leq&\frac{1}{\theta}\sum_{k=0}^{n-1}e^{-\theta k\Delta_n}\nonumber\\
 &\leq&\frac{1-e^{-\theta n\Delta_n}}{\theta(1-e^{-\theta \Delta_n})}\nonumber\\
 &\leq&\frac{C}{\Delta_n}.\label{k3-estimate2}
\end{eqnarray}
Combining \eqref{k3-estimate1} and \eqref{k3-estimate2} yields
\begin{eqnarray*}
k_{3}(F_n(Z))&\leq&\frac{C\Delta_n^{1/2}}{n^{3/2}},
\end{eqnarray*}
which implies \eqref{k3-estimator1}.\\
Using \eqref{4th-cumulant} and \eqref{F_n(Z) in 2nd chaos}, we get
\begin{eqnarray*} \left|k_{4}(F_n(Z))\right|
&\leq&48\|\varepsilon_n\otimes_1
\varepsilon_n\|_{\mathcal{H}^{\otimes
2}}^2\\&=&48\frac{\Delta_n^2}{n^{2}}\sum_{k_{1},k_{2},k_{3},k_{4}=0}^{n-1}
\inner{\stepid{t_{k_{1}}}^{\otimes 2} \otimes_1
\stepid{t_{k_{2}}}^{\otimes 2}}
{\stepid{t_{k_{3}}}^{\otimes 2} \otimes_1 \stepid{t_{k_{4}}}^{\otimes 2}}_{\HC^{\otimes 2}}\\
&=&48\frac{\Delta_n^2}{n^{2}}\sum_{k_{1},k_{2},k_{3},k_{4}=0}^{n-1}
E[ Z_{t_{k_{1}}} Z_{t_{k_{2}}}]E[ Z_{t_{k_{3}}} Z_{t_{k_{4}}}]E[
Z_{t_{k_{1}}} Z_{t_{k_{3}}}]E[ Z_{t_{k_{2}}} Z_{t_{k_{4}}}]
 \\
&=&48\frac{\Delta_n^2}{n^{2}}\sum_{k_{1},k_{2},k_{3},k_{4}=0}^{n-1}
\rho(t_{k_{1}}-t_{k_{2}}) \rho(t_{k_{3}}-t_{k_{4}})
\rho(t_{k_{1}}-t_{k_{3}}) \rho(t_{k_{2}}-t_{k_{4}}),
\end{eqnarray*}
where we used
\begin{align}
\notag \stepid{s}^{\otimes 2} \otimes_1 \stepid{t}^{\otimes 2} =&
\inner{\stepid{s}}{\stepid{t}}_\HC \stepid{s} \otimes \stepid{t}\\=&
E[ Z_s Z_t]\stepid{s} \otimes \stepid{t}.
\end{align}
Furthermore,

\begin{eqnarray}&&48\frac{\Delta_n^2}{n^{2}}\sum_{k_{1},k_{2},k_{3},k_{4}=0}^{n-1}
\rho(t_{k_{1}}-t_{k_{2}}) \rho(t_{k_{3}}-t_{k_{4}})
\rho(t_{k_{1}}-t_{k_{3}}) \rho(t_{k_{2}}-t_{k_{4}})\nonumber
\\
&= &48\frac{\Delta_n^2}{n^{2}}
\sum_{k_{1},k_{2},k_{3},k_{4}=0}^{n-1}\rho((k_{1}-k_{2})\Delta_n)\rho((k_{3}-k_{4})\Delta_n)
\rho((k_{1}-k_{3}))\Delta_n)\rho((k_{2}-k_{4}))\Delta_n)
\nonumber\\
&=&48\frac{\Delta_n^2}{n}
\sum_{\underset{i=1,2,3}{\left|j_{i}\right| <
n}}\left|\rho\left(j_{1}\Delta_n\right)
\rho\left(j_{2}\Delta_n\right)
\rho\left(j_{3}\Delta_n\right) \rho\left((j_{1}+j_{2}-j_{3})\Delta_n\right)\right|\nonumber\\
&\leq&C\frac{\Delta_n^2}{n}\left(\sum_{|k| <
n}|\rho(k\Delta_n)|^{\frac{4}{3}}\right)^{3}\nonumber\\
&\leq&C\frac{1}{n\Delta_n}\left(\Delta_n\sum_{|k| <
n}|\rho(k\Delta_n)|^{\frac{4}{3}}\right)^{3}
\nonumber\\
&\leq&C\frac{1}{n\Delta_n},
\end{eqnarray}%
where we used the the change of variables $k_{1}-k_{2}=j_{1},
k_{2}-k_{4}=j_{2}$ and $k_{3}-k_{4}=j_{3}$, and then applying
Brascamp-Lieb inequality given by Lemma \ref{NZ-lemma}. Therefore
the proof  of \eqref{k4-estimator1} is complete.
\end{proof}

\begin{theorem}\label{rate-CLT-Fn}
There exists $C>0$ depending only on $\theta$ such that for all
$n\geq1$, \begin{eqnarray*}
d_W\left(\sqrt{2}\theta^{3/2}F_n(X),\mathcal{N}\right) &\leq& C
\left(\Delta_n^2+\frac{1}{n\Delta_n}\right).
\end{eqnarray*}
\end{theorem}
\begin{proof}
Using \eqref{error X-Z} and \eqref{variance-F(Z)}, we obtain
\begin{eqnarray*}
&&d_W\left(\sqrt{2}\theta^{3/2}F_n(X),\mathcal{N}\right)\\
&\leq&d_W\left(\sqrt{2}\theta^{3/2}F_n(Z),\mathcal{N}\right)+\left\|F_n(X)-F_n(Z)\right\|_{L^2(\Omega)}\\
&\leq& d_W\left(\frac{F_n(Z)}{\sqrt{E(F_n^2(Z))}},\mathcal{N}\right)
+{E}\left|\frac{\sqrt{2}\theta^{3/2}F_n(Z)}{\sqrt{E(F_n^2(Z))}}\left(\frac{1}{\sqrt{2}\theta^{3/2}}-\sqrt{E(F_n^2(Z))}\right)\right|
+\frac{C}{n\Delta_n}\\&\leq&
d_W\left(\frac{F_n(Z)}{\sqrt{E(F_n^2(Z))}},\mathcal{N}\right)
+\left|E\left(F_n^2(Z)\right)-\frac{1}{2\theta^3}\right|
+\frac{C}{n\Delta_n}
\\&\leq&
d_W\left(\frac{F_n(Z)}{\sqrt{E(F_n^2(Z))}},\mathcal{N}\right) +C
\left(\Delta_n^2+\frac{1}{n\Delta_n}\right)\\
&\leq& C \left(\Delta_n^2+\frac{1}{n\Delta_n}\right),
\end{eqnarray*}
where the latter inequality comes from \eqref{fourth-cumulant-thm}
and \eqref{max(k3,k4)-estimator1}.
\end{proof}

\begin{theorem}Suppose $\Delta_n\rightarrow0$ and $T\rightarrow\infty$. Then, the estimator   $\widetilde{\theta}_{n}$ of $\theta$ is weakly
consistent, that is, $\widetilde{\theta}_{n}\rightarrow \theta$ in
probability, as $\Delta_n\rightarrow0$ and $T\rightarrow\infty$.\\
If, moreover, $n \Delta_n^\eta \rightarrow0$ for some $1<\eta<2$ or
$n \Delta_n^\eta \rightarrow\infty$ for some $\eta>1$, then
$\widetilde{\theta}_{n}$ is strongly consistent, that is,
$\widetilde{\theta}_{n}\rightarrow \theta$ almost surely.
\end{theorem}
\begin{proof}Using \eqref{expr-theta-tilde}, it is sufficient to
prove that the results of the theorem are satisfied for the
estimator
$f_n(X)$  of $\frac{1}{2\theta}$ .\\
The weak consistency of  $f_n(X)$ is an immediate
consequence from \eqref{variance-F(X)}.\\
If  $n \Delta_n^\eta \rightarrow0$ for some  $1<\eta<2$, the strong
consistency of $f_n(X)$ has been proved by \cite[Theorem 11]{HNZ}.\\
Now, suppose that $n \Delta_n^\eta \rightarrow\infty$ for some
$\eta>1$. It follows from \eqref{variance-F(X)} that

$$
E\left[\left(f_n(X)-\frac{1}{2\theta}\right)^2\right] \leq
\frac{C}{n\Delta_n} \leq
  \frac{C}{n^{1- 1 / \eta }\left(n\Delta_n^{\eta }\right)^{1/
\eta }} \leq   \frac{C}{n^{1- 1 / \eta }}.
$$
Combining this with the hypercontractivity property
\eqref{hypercontractivity} and \cite[Lemma 2.1]{KN}, which is a
well-known direct consequence of the Borel-Cantelli Lemma, we obtain
$f_n(X)\rightarrow\frac{1}{2\theta}$ almost surely.
\end{proof}
\begin{theorem}\label{rate-CLT-theta-tilde}
There exists $C>0$ depending only on $\theta$ such that for all
$n\geq1$,
\begin{eqnarray}
d_{W}\left(\sqrt{\frac{T}{2\theta}}\left(\widetilde{\theta}_{n}-\theta\right),
\mathcal{N}\right) \leq  C
\left(\Delta_n^2+\frac{1}{\sqrt{n\Delta_n}}\right).
\label{eq:bound_tilde}\end{eqnarray}
\end{theorem}
\begin{proof}
 Recall that by definition
$\theta=g\left(\frac{1}{2\theta}\right)$. We have
$$
\left(\widetilde{\theta}_{n}-\theta\right)=\left(g(f_{n}\left(X\right))-g\left(\frac{1}{2\theta}\right)\right)=g^{\prime}\left(\frac{1}{2\theta}\right)\left(f_{n}\left(X\right)-\frac{1}{2\theta}\right)+\frac{1}{2}
g^{\prime
\prime}\left(\zeta_{n}\right)\left(f_{n}\left(X\right)-\frac{1}{2\theta}\right)^{2}
$$
for some random point $\zeta_{n}$ between $f_{n}\left(X\right)$ and
$\frac{1}{2\theta}$.\\
Thus, we can write,
$$
\sqrt{\frac{T}{2\theta}}\left(\widetilde{\theta}_{n}-\theta\right)=-\sqrt{2}\theta^{3/2}F_{n}\left(X\right)+\frac{1}{2^{3/2}\sqrt{\theta
T}\zeta_{n}^3} \left(F_{n}\left(X\right)\right)^{2}.
$$
Therefore,
$$
d_{W}\left(\sqrt{\frac{T}{2\theta}}\left(\widetilde{\theta}_{n}-\theta\right),
\mathcal{N}\right) \leq
 \frac{1}{2^{3/2}\sqrt{\theta
T}} {E}\left|\frac{1}{\zeta_{n}^3}
\left(F_{n}\left(X\right)\right)^{2}\right|
+d_{W}\left(\sqrt{2}\theta^{3/2}F_{n}\left(X\right),
\mathcal{N}\right),
$$
where we have used that $d_{W}\left(x_{1}+x_{2}, y\right) \leq
{E}\left[\left|x_{2}\right|\right]+d_{W}\left(x_{1}, y\right)$ for
any random variables $x_{1}, x_{2}, y$.\\
 The second term in the
inequality above is bounded in Theorem \ref{rate-CLT-Fn}. By
H\"older's inequality, and the hypercontractivity property
\eqref{hypercontractivity}, for $p, q>1$ with $1 / p+$ $1 / q=1$
$$
\begin{aligned}
{E}\left|\frac{1}{\zeta_{n}^3}
\left(F_{n}\left(X\right)\right)^{2}\right| &
\leq\left({E}\left|\frac{1}{\zeta_{n}^3}\right|^{p}\right)^{1 / p}
\left({E}\left|F_{n}\left(X\right)\right|^{2 q}\right)^{1 / q} \\
& \leq
c_{p,q}\left({E}\left|\frac{1}{\zeta_{n}^3}\right|^{p}\right)^{1/p}
{E}\left|F_{n}\left(X\right)\right|^{2},\\
& \leq C
\left({E}\left|\frac{1}{\zeta_{n}^3}\right|^{p}\right)^{1/p},
\end{aligned}
$$
for some constant $C>0$ depending on $p$.\\
Consequently, for every $p\geq1$
$$
d_{W}\left(\sqrt{\frac{T}{2\theta}}\left(\widetilde{\theta}_{n}-\theta\right),
\mathcal{N}\right) \leq
 C
\left({E}\left|\frac{1}{\zeta_{n}^3}\right|^{p}\right)^{1/p} + C
\left(\Delta_n^2+\frac{1}{n\Delta_n}\right).
$$
To establish \eqref{eq:bound_tilde} it is left to show that
${E}\left|\zeta_{n}\right|^{-3p} < \infty$ for some $p \geq 1$.
Using the monotonocity of $x^{-3}$ and the fact
 that $\zeta_n \in [|f_{n}\left(X\right), \frac{1}{2\theta}|]$, it is enough to
 show that $E |f_{n}\left(X\right)|^{-3p} < \infty$ for some $p \geq 1$.
 This follows as an application of the technical \cite[Proposition 6.3]{DEKN}.
\end{proof}

\section{Approximate maximum likelihood estimator}

The maximum likelihood estimator  for $\theta$ based on continuous
observations of the process $X$ given by \eqref{OU}, is defined by
\begin{eqnarray}
 \check{\theta}_{T}=\frac{\int_{0}^{T} X_{s} \mathrm{~d} X_{s}}{\int_{0}^{T}
X_{s}^{2} \mathrm{~d} s}, \quad T\geq 0. \label{MLE-cont}
\end{eqnarray}

Here we want to study  the asymptotic distribution of a discrete
version of \eqref{MLE-cont}. Then, we assume that the process $X$
given in \eqref{OU} is observed equidistantly in time with the step
size $\Delta_{n}$ : $t_{i}=i \Delta_{n}, i=0, \cdots, n$, and $T=n
\Delta_{n}$ denotes the length of the "observation window". Let us
consider the following discrete version of $\check{\theta}_{T}$:
$$
\widehat{\theta}_{n}=-\frac{\sum_{i=1}^{n}
X_{t_{i-1}}\left(X_{t_{i}}-X_{t_{i-1}}\right)}{\Delta_{n}
\sum_{i=1}^{n} X_{t_{i-1}}^{2}},\quad n\geq1.
$$

Note that  \cite{dorogovcev} and  \cite{kasonga}, respectively,
proved the weak and strong consistency of the estimator
$\widehat{\theta}_{n}$ as $T \rightarrow \infty$ and $\Delta_n
\rightarrow 0$.

Let $X$ be the process given by \eqref{OU}, and let us introduce the
following sequences
$$
S_{n}:=\Delta_{n} \sum_{i=1}^{n} X_{t_{i-1}}^{2},
$$
and
$$
\Lambda_{n}:=\sum_{i=1}^{n} e^{-\theta
t_{i}}X_{t_{i-1}}\left(\zeta_{t_{i}}-\zeta_{t_{i-1}}\right)=\sum_{i=1}^{n}
e^{-\theta(t_{i}+t_{i-1})}\zeta_{t_{i-1}}\left(\zeta_{t_{i}}-\zeta_{t_{i-1}}\right),
$$
where
\[\zeta_{t}=\int_{0}^{t}e^{\theta s}dW_{s}. \label{zeta}\]
Thus,
$$
-\widehat{\theta}_{n}=\frac{e^{-\theta
\Delta_{n}}-1}{\Delta_{n}}+\frac{\Lambda_{n}}{S_{n}}.
$$
Therefore
\begin{eqnarray}
\sqrt{T}\left(\theta-\widehat{\theta}_{n}\right)&=&\sqrt{T}\left(\frac{e^{-\theta
\Delta_{n}}-1}{\Delta_{n}}+\theta\right)+\frac{\frac{1}{\sqrt{T}}
\Lambda_{n}}{\frac{1}{T} S_{n}}\nonumber\\
&=&\sqrt{T}\left(\frac{e^{-\theta
\Delta_{n}}-1}{\Delta_{n}}+\theta\right)+\frac{\frac{1}{\sqrt{T}}
\Lambda_{n}}{f_{n}(X)}\nonumber\\
&=&\sqrt{T}\left(\frac{\theta^2}{2}\Delta_{n}+o(\Delta_{n})\right)+\frac{\frac{1}{\sqrt{T}}
\Lambda_{n}}{f_{n}(X)}\nonumber \\
&=&\sqrt{n\Delta_{n}^3}\left(\frac{\theta^2}{2}+o(1)\right)+\frac{\frac{1}{\sqrt{T}}
\Lambda_{n}}{f_{n}(X)},\label{theta-hat-decomp}
 \end{eqnarray}
 where $f_n(X)$ is given by \eqref{def-f_n}.\\
Next, since $\zeta_{t_{i-1}}$ and $\zeta_{t_{i}}-\zeta_{t_{i-1}}$
are independent, we have
\begin{eqnarray*}
E\left[\left(\frac{1}{\sqrt{T}}
\Lambda_{n}\right)^2\right]&=&\frac{1}{T}\sum_{i,j=1}^{n}
e^{-\theta(t_{i}+t_{i-1}+t_{j}+t_{j-1})}
E\left[\zeta_{t_{i-1}}\left(\zeta_{t_{i}}-\zeta_{t_{i-1}}\right)\zeta_{t_{j-1}}\left(\zeta_{t_{j}}-\zeta_{t_{j-1}}\right)\right]\\
&=&\frac{1}{T}\sum_{i=1}^{n} e^{-2\theta(t_{i}+t_{i-1})}
E\left[\zeta_{t_{i-1}}^2\left(\zeta_{t_{i}}-\zeta_{t_{i-1}}\right)^2\right]
\\
&=&\frac{1}{T}\sum_{i=1}^{n} e^{-2\theta(t_{i}+t_{i-1})}
E\left[\zeta_{t_{i-1}}^2\right]
E\left[\left(\zeta_{t_{i}}-\zeta_{t_{i-1}}\right)^2\right]\\
&=&\frac{1}{T}\sum_{i=1}^{n} e^{-2\theta(t_{i}+t_{i-1})}
\left(\frac{e^{2\theta  t_{i-1}}-1}{2\theta}\right)
\left(\frac{e^{2\theta  t_{i}}-e^{2\theta t_{i-1}}}{2\theta}\right)\\
&=&\frac{\left(1-e^{-2\theta
\Delta_{n}}\right)}{(2\theta)^2\Delta_{n}} \frac{1}{n}\sum_{i=1}^{n}
\left(1-e^{-2\theta t_{i-1}}\right)\\
&=&\frac{\left(1-e^{-2\theta
\Delta_{n}}\right)}{(2\theta)^2\Delta_{n}}
-\frac{\left(1-e^{-2\theta
\Delta_{n}}\right)}{(2\theta)^2\Delta_{n}}\left(\frac{1-e^{-2\theta
T}}{n(1-e^{-2\theta \Delta_n})}\right).
\end{eqnarray*}
Moreover, since
\[\frac{\left(1-e^{-2\theta\Delta_{n}}\right)}{(2\theta)^2\Delta_{n}}=\frac{1}{2\theta}-\frac{\Delta_n}{2}+o(\Delta_n),\]
there exists $C>0$ depending only on $\theta$ such that for large
$n$
\begin{align}
\left|E\left[\left(\frac{1}{\sqrt{T}}
\Lambda_{n}\right)^2\right]-\frac{1}{2\theta}\right|&\leq
C\left(\Delta_n+\frac{1}{n\Delta_n}\right).\label{cv-Lambda}
\end{align}
Using $E[\Lambda_{n}]=0$ and the fact  that   $\zeta_{t_{i-1}}$ and
$\zeta_{t_{i}}-\zeta_{t_{i-1}}$ are independent, we get
\begin{eqnarray}\kappa_{3}\left(\frac{1}{\sqrt{T}}
\Lambda_{n}\right)=E\left[\left(\frac{1}{\sqrt{T}}
\Lambda_{n}\right)^3\right]=0.\label{k3-Gn} \end{eqnarray}
 On the other hand,
\begin{eqnarray*}
E\left[\left(\frac{1}{\sqrt{T}}
\Lambda_{n}\right)^4\right]&=&\frac{1}{T^2}\sum_{i,j,k,l=1}^{n}
e^{-\theta(t_{i}+t_{i-1}+t_{j}+t_{j-1}+t_{k}+t_{k-1}+t_{l}+t_{l-1})}\\
&&\quad\times
E\left[\zeta_{t_{i-1}}\left(\zeta_{t_{i}}-\zeta_{t_{i-1}}\right)\zeta_{t_{j-1}}\left(\zeta_{t_{j}}-\zeta_{t_{j-1}}\right)
\zeta_{t_{k-1}}\left(\zeta_{t_{k}}-\zeta_{t_{k-1}}\right)\zeta_{t_{l-1}}\left(\zeta_{t_{l}}-\zeta_{t_{l-1}}\right)\right]\\
&=&\frac{1}{T^2}\sum_{i=1}^{n} e^{-4\theta(t_{i}+t_{i-1})}
E\left[\zeta_{t_{i-1}}^4\left(\zeta_{t_{i}}-\zeta_{t_{i-1}}\right)^4\right]\\
&&+\frac{3}{T^2}\sum_{i=j\neq k=l}^{n}
e^{-2\theta(t_{i}+t_{i-1}+t_{k}+t_{k-1})}
E\left[\zeta_{t_{i-1}}^2\left(\zeta_{t_{i}}-\zeta_{t_{i-1}}\right)^2\zeta_{t_{k-1}}^2\left(\zeta_{t_{k}}-\zeta_{t_{k-1}}\right)^2\right]
\\
&=&\frac{6}{T^2}\sum_{i=1}^{n} e^{-4\theta(t_{i}+t_{i-1})}
\left(E\left[\zeta_{t_{i-1}}^2\right]\right)^2\left(E\left[\left(\zeta_{t_{i}}-\zeta_{t_{i-1}}\right)^2\right]\right)^2\\
&&+3\left[\frac{1}{T}\sum_{i=1}^{n} e^{-2\theta(t_{i}+t_{i-1})}
E\left[\zeta_{t_{i-1}}^2\left(\zeta_{t_{i}}-\zeta_{t_{i-1}}\right)^2\right]\right]^2\\
\\
&=&\frac{6}{T^2}\sum_{i=1}^{n} e^{-4\theta(t_{i}+t_{i-1})}
\left(E\left[\zeta_{t_{i-1}}^2\right]\right)^2\left(E\left[\left(\zeta_{t_{i}}-\zeta_{t_{i-1}}\right)^2\right]\right)^2
+3\left[E\left[\left(\frac{1}{\sqrt{T}}
\Lambda_{n}\right)^2\right]\right]^2.
\end{eqnarray*}
This implies
\begin{eqnarray}\kappa_{4}\left(\frac{1}{\sqrt{T}}
\Lambda_{n}\right)&=&E\left[\left(\frac{1}{\sqrt{T}}
\Lambda_{n}\right)^4\right]-3\left[E\left[\left(\frac{1}{\sqrt{T}}
\Lambda_{n}\right)^2\right]\right]^2\nonumber\\
&=&\frac{6}{T^2}\sum_{i=1}^{n} e^{-4\theta(t_{i}+t_{i-1})}
\left(E\left[\zeta_{t_{i-1}}^2\right]\right)^2\left(E\left[\left(\zeta_{t_{i}}-\zeta_{t_{i-1}}\right)^2\right]\right)^2\nonumber
\\
&=&\frac{6}{T^2}\sum_{i=1}^{n} e^{-4\theta(t_{i}+t_{i-1})}
\left(\frac{e^{2\theta  t_{i-1}}-1}{2\theta}\right)^2
\left(\frac{e^{2\theta  t_{i}}-e^{2\theta
t_{i-1}}}{2\theta}\right)^2\nonumber\\
&=&\frac{\left(1-e^{-2\theta
\Delta_{n}}\right)^2}{(2\theta)^4\Delta_{n}^2}
\frac{1}{n^2}\sum_{i=1}^{n}
\left(1-e^{-2\theta t_{i-1}}\right)^2\nonumber\\
&\leq&\frac{\left(1-e^{-2\theta
\Delta_{n}}\right)^2}{(2\theta)^4\Delta_{n}^2} \frac{1}{n}\nonumber
\\
&\leq&\frac{C}{n},\label{k4-Gn}
\end{eqnarray}
where the latter inequality comes from the fact that
$\displaystyle{\frac{1-e^{-2\theta\Delta_n}}{\Delta_n}}\rightarrow2\theta$
as $n\rightarrow\infty$.

\begin{theorem}\label{thm:sigma-hat}  There exists a constant $C>0$ such that, for all $n\geq1$,

\begin{eqnarray}
 d_W\left(\frac{\frac{1}{2\theta}\sqrt{T}}{\sqrt{E(G_n^2)}}\left(\widehat{\theta}_{n}-\theta\right),\mathcal{N}\right)
 &\leq&C\left(\frac{1}{\sqrt{n\Delta_n}}+\sqrt{n\Delta_{n}^3}\right).\label{rate1-hat}
\end{eqnarray}
Moreover,
\begin{eqnarray}
 d_W\left(\sqrt{\frac{T}{2\theta}}\left(\widehat{\theta}_{n}-\theta\right),\mathcal{N}\right)
 &\leq&C\left(\frac{1}{\sqrt{n\Delta_n}}+C\sqrt{n\Delta_{n}^3}\right).\label{rate2-hat}
\end{eqnarray}
\end{theorem}

\begin{proof}
Define $G_n:=\frac{1}{\sqrt{T}} \Lambda_{n}$. Using
\eqref{fourth-cumulant-thm}, \eqref{k3-Gn} and \eqref{k4-Gn}, we
have
\begin{eqnarray}d_W\left(\frac{G_n}{\sqrt{E(G_n^2)}},\mathcal{N}\right)\leq\frac{C}{n}.\label{rate-Gn}
\end{eqnarray}
Combining \eqref{rate-Gn}   with \eqref{theta-hat-decomp},
\eqref{cv-Lambda} and \eqref{variance-F(X)}, we obtain
\begin{eqnarray*}
&&d_W\left(\frac{\frac{1}{2\theta}\sqrt{T}}{\sqrt{E(G_n^2)}}\left(\theta-\widehat{\theta}_{n}\right),\mathcal{N}\right)\\
&\leq&d_W\left(\frac{\frac{1}{2\theta}}{\sqrt{E(G_n^2)}}\frac{G_n}{f_{n}(X)},\mathcal{N}\right)+C\sqrt{n\Delta_{n}^3}
\\&\leq&
d_W\left(\frac{G_n}{\sqrt{E(G_n^2)}},\mathcal{N}\right)+{E}\left|\frac{G_n}{\sqrt{E(G_n^2)}}\left(\frac{\frac{1}{2\theta}}{f_{n}(X)}-1\right)\right|
+C\sqrt{n\Delta_{n}^3}
\\&\leq&d_W\left(\frac{G_n}{\sqrt{E(G_n^2)}},\mathcal{N}\right)+
\left\|\frac{G_n}{\sqrt{E(G_n^2)}}\right\|_{L^4(\Omega)}\left\|\frac{1}{f_{n}(X)}\right\|_{L^4(\Omega)}
\left\|f_{n}(X)-\frac{1}{2\theta}\right\|_{L^2(\Omega)}+C\sqrt{n\Delta_{n}^3}\nonumber
\\&\leq&C\left(\frac{1}{n}+\frac{1}{\sqrt{n\Delta_n}}\right)+C\sqrt{n\Delta_{n}^3}
\\&\leq&C\left(\frac{1}{\sqrt{n\Delta_n}}+\sqrt{n\Delta_{n}^3}\right),
\end{eqnarray*}
where   we used the fact that $E |f_{n}\left(X\right)|^{-4} <
\infty$, which is a direct application of the technical
\cite[Proposition 6.3]{DEKN}. Therefore,
\eqref{rate1-hat} is obtained.\\
Similarly,
\begin{eqnarray*}
&&d_W\left(\sqrt{\frac{T}{2\theta}}\left(\widehat{\theta}_{n}-\theta\right),\mathcal{N}\right)\\
&\leq&d_W\left( \frac{1}{\sqrt{2\theta}}
\frac{G_n}{f_{n}(X)},\mathcal{N}\right)+C\sqrt{n\Delta_{n}^3}
\\&\leq&
d_W\left(\frac{G_n}{\sqrt{E(G_n^2)}},\mathcal{N}\right)+{E}\left|\frac{G_n}{\sqrt{E(G_n^2)}f_{n}(X)}\left(\frac{\sqrt{E(G_n^2)}}{\sqrt{2\theta}}-f_{n}(X)\right)\right|
+C\sqrt{n\Delta_{n}^3}
\\&\leq&d_W\left(\frac{G_n}{\sqrt{E(G_n^2)}},\mathcal{N}\right)+
\left\|\frac{G_n}{\sqrt{E(G_n^2)}}\right\|_{L^4(\Omega)}\left\|\frac{1}{f_{n}(X)}\right\|_{L^4(\Omega)}
\left\|\frac{\sqrt{E(G_n^2)}}{\sqrt{2\theta}}-f_{n}(X)\right\|_{L^2(\Omega)}+C\sqrt{n\Delta_{n}^3}\nonumber
\\&\leq&C\left(\frac{1}{n}+\frac{1}{\sqrt{n\Delta_n}}\right)+C\sqrt{n\Delta_{n}^3}
\\&\leq&C\left(\frac{1}{\sqrt{n\Delta_n}}+\sqrt{n\Delta_{n}^3}\right),
\end{eqnarray*}
which  proves \eqref{rate2-hat}.

\end{proof}

\vspace{1cm}

\noindent\textbf{Funding}\\
 This project was funded by Kuwait
Foundation for the Advancement of Sciences (KFAS) under project
code: PR18-16SM-04.\\

\end{document}